\documentclass[en-us,11pt]{amsart}

\usepackage{amssymb,amsmath,amsthm,mathrsfs,amsfonts}

\usepackage{microtype}

\usepackage[titletoc,toc, title]{appendix}
\usepackage[colorlinks=true,linkcolor=blue,citecolor=blue]{hyperref}
\usepackage{tikz}
\usetikzlibrary{matrix,fit,positioning}

\usepackage{xcolor}

\usepackage[OT1]{fontenc}
\usepackage[utf8]{inputenc}

\numberwithin{equation}{section}
\newtheorem{theorem}{Theorem}[section]
\newtheorem{lemma}[theorem]{Lemma}

\newtheorem{corollary}[theorem]{Corollary}
\newtheorem{claim}[theorem]{Claim}

\newtheorem*{maintheorem}{Main Theorem}

\theoremstyle{definition}
\newtheorem{definition}[theorem]{Definition}
\newtheorem*{notation}{Notation}
\newtheorem{remark}[theorem]{Remark}
\theoremstyle{remark}

\newtheorem{example}[theorem]{Example}

\newcommand{\init}{\operatorname{in}}
\newcommand{\functorslot}{\underline{\hspace{1ex}}}
\newcommand{\height}{\operatorname{height}}
\newcommand{\Ext}{\operatorname{Ext}}

\renewcommand{\geq}{\geqslant}
\renewcommand{\leq}{\leqslant}

\begin{document}
\title[Double determinantal varieties]{Gr\"obner bases and the Cohen-Macaulay property of Li's double determinantal varieties}

\author[Nathan Fieldsteel]{Nathan Fieldsteel}
\address{Department of Mathematics, University of Kentucky, Lexington, KY 40506 USA}
\email{nathan.fieldsteel@uky.edu}

\author[Patricia Klein]{Patricia Klein}
\address{Department of Mathematics, University of Minnesota, Minneapolis, MN 55455 USA}
\email{klein847@umn.edu}

\maketitle

\begin{abstract}
We consider double determinantal varieties, a special case of Nakajima quiver varieties.  Li conjectured that double determinantal varieties are normal, irreducible, Cohen-Macaulay varieties whose defining ideals have a Gr\"obner basis given by their natural generators.  We use liaison theory to prove this conjecture in a manner that generalizes results for mixed ladder determinantal varieties.  We also give a formula for the dimension of a double determinantal variety.
\end{abstract}

\section{Introduction}

In order to study Kac-Moody Lie algebras, Nakajima introduced what are now called Nakajima quiver varieties, which are quantized enveloping algebras of Kac-Moody Lie algebras and can be defined in terms of an appropriate quiver.  These quiver varieties have been shown to provide a powerful perspective for bringing geometric tools to bear on representation theory.  For an introduction to the construction and uses of Nakajima quiver varieties, see \cite{Nak16} or \cite{Gin08}.  However, even small examples of Nakajima quiver varieties can be very challenging to study.  

Recently, Li defined \emph{double determinantal varieties}, a special case of Nakajima quiver varieties.  A double determinantal variety is defined by the vanishing of minors of some size $s$ in a concatenation of finitely many $m \times n$ matrices glued along their size $m$-edges together with the vanishing of minors of a possibly different size $t$ in a concatenation of the same matrices along their length $n$ edges (see Definition \ref{doubleDeterminantalVarieties}).  For an explanation of how double determinantal varieties are instances of Nakajima quiver varieties, see \cite{IL}.  Li conjectured that with respect to any diagonal term order (see Definition \ref{diagonal}), the natural generators of the ideal defining a double determinantal variety form a Gr\"obner basis and that double determinantal varieties are irreducible, normal, and Cohen-Macaulay.  

Although Li's conjecture was motivated by the combinatorics and geometry surrounding Nakajima quiver varieties, it also has a natural home in the commutative algebra literature, in particular the literature on \emph{mixed ladder determinantal varieties}, which are defined by the vanishing of determinants of varying sizes in a matrix of indeterminates after possibly excluding some variables from one corner or two opposite corners in a ladder shape.  

Determinantal varieties and their generalizations to ladder determinantal varieties exhibit close connections to geometry, representation theory, and combinatorics.  Classical determinantal rings, by which we mean quotients of polynomial rings by minors of a fixed size in one matrix of indeterminates, arose in the context of invariant theory and were first shown to be normal Cohen-Macaulay domains in 1971 by Hochster and Eagon \cite{HE71}.  Ladder determinantal varieties arose in Abhyankar's study of Schubert varieties of flag manifolds.  Ladder determinantal rings are known to be normal \cite{Con95} Cohen-Macaulay \cite{HT92} domains \cite{Nar86}.  One-sided ladder determinantal varieties are also known, primarily in the combinatorics literature, as vexillary matrix Schubert varieties.  A further generalization, mixed ladder determinantal varieties, were shown by Gonciulea and Miller \cite{GM00} to correspond to opposite cells in Schubert varieties in flag varieties of a certain type, from which they infer normality and Cohen-Macaulyness. Using proof techniques we will borrow in this paper, Gorla \cite{Gor07} showed that mixed ladder determinantal varieties are irreducible, arithmetically Cohen--Macaulay, and glicci (see Definition \ref{glicci}) and characterized when they are arithmetically Gorenstein.  Gorla used a less restrictive definition of mixed ladder determinantal variety than did Gonciulea and Miller, and it is Gorla's notion that we will use in this paper (see Definition \ref{mixedLadder}).  Connections between these varieties and liaison theory was studied further in \cite{GMN13}, whose key lemma appears at the top of our Section \ref{specialsection} and is used throughout the current work.

However, the literature on ladder determinantal ideals always insists that a ladder be considered within a matrix of distinct indeterminates or a symmetric or skew-symmetric matrix of indeterminates. The case of double determinantal varieties does not have any of these structures, and so new strategies are required.  In this article, we adapt techniques historically used to study mixed ladder determinantal ideals in order to prove Li's conjecture.  The main theorem of this article is \begin{maintheorem}\label{maintheorem}
The natural generators of a double determinantal ideal form a Gr\"obner basis under any diagonal term order.  Double determinantal varieties are irreducible, normal, and arithmetically Cohen-Macaulay.  In addition, they are glicci.
\end{maintheorem} 

A different, more elementary proof of Li's conjecture, due to Li and Illian, precedes ours and is currently in preparation for publication.  The result in the  special case of $2$-minors and 2 matrices, due to Conca, De Negri, and Stojanac, appears in a forthcoming paper studying the secant varieties of the triple Segre product \cite{CDS}.

This paper is structured as follows: Section \ref{specialsection} contains a proof of a special case of the Cohen-Macaulay and reduced components of the Main Theorem.  This special case is explained in full detail without explicit reference to liaison theory.  It is intended to be readable to a mathematician with basic familiarity with Hilbert functions and the Cohen-Macaulay property but possibly none with liaison theory.  The goal is for this special case to serve as an example of how to use the style of induction common in liaison theory arguments, which is a natural approach for questions involving determinants, for those not desiring to learn the full machinery of liaison theory or to take theorems from that literature on faith.  A key technical lemma that appears in Section \ref{specialsection} is used throughout the paper.  Section \ref{generalsection} contains a proof of the general case of this paper's Main Theorem using the machinery of liaison theory directly.  It also contains a dimension formula for double determinantal varieties.   

\bigskip

\noindent\textbf{Acknowledgements}: The authors are grateful to Uwe Nagel, Oliver Pechenik, Jenna Rajchgot, Aldo Conca, and Claudia Miller for helpful conversations.  They are also grateful to Li Li for very helpful communications, including comments on a previous draft of this paper.  

\section{Preliminaries}
\label{prelim}

We use this section to record some definitions that we will need throughout this paper.  We provide references for a broad background on Hilbert functions (\cite[Section 13]{Mat89}), the Cohen--Macaulay property (\cite[Section 17] {Mat89}), and Gr\"obner bases (\cite{CLO94}).  

Our primary objects of study in this paper are Li's double determinantal varieties.

\begin{definition}\label{doubleDeterminantalVarieties}
  Fix $r, m, n \geq 1$, and let $X_q = (x^q_{i,j})$ with $1 \leq q \leq r$ be $m \times n$ matrices of distinct indeterminates.  With $[n] = \{1, \ldots, n\}$, let $R = K[x^q_{i,j} \mid i \in [m], j \in [n], q \in [r]]$ be the standard graded polynomial ring in the indeterminates that appear in the matrices $X_q$ over the perfect field $K$. Let \(H\) be the horizontal concatenation of these matrices, i.e., the \(m \times rn\) matrix

  \[H =
    \begin{pmatrix}
      X_1 \cdots X_r\\
    \end{pmatrix},
  \]

  and let \(V\) be their vertical concatenation, i.e., the \(rm \times n\) matrix

  \[V =
    \begin{pmatrix}
      X_{1} \\
      \vdots \\
      X_{r} \\
    \end{pmatrix}.  
  \]

  The ideal $J$ generated by the $s$-minors of $H$ together with the $t$-minors of $V$ is called a \textit{double determinantal ideal}, and the variety cut out by $J$ is called a \emph{double determinantal variety}.  Because $J$ homogeneous, we may think of the variety it defines as either an affine variety or a projective variety.  
  
  \end{definition}
  
  Note that part of the Main Theorem is that $J$ is normal, which implies reduced, and so the vanishing locus of \(J\) is indeed a variety.  
  
  Our first step in approaching the Main Theorem will be to show that the natural generators of a double determinantal ideal form a Gr\"obner basis under any diagonal term order.
  
  \begin{definition}\label{diagonal}
 In the context of a double determinantal ideal $J$, we call a term order $\sigma$ \emph{diagonal} if for every square submatrix of $H$ or of $V$ the leading term of its determinant with respect to $\sigma$ is the product of the entries along its main diagonal.
 \end{definition}

Such term orders exist.  One example is the graded lexicographic ordering induced by the reading order on the entries of $V$.  

\begin{definition}\label{mixedLadder}
Let $X = (x_{i,j})$ be an $m \times n$ matrix of indeterminates, and fix $h < \min\{m,n\}$, and let $R = K[x_{i,j} \mid i \in [m], j \in [n]]$ be the corresponding polynomial ring over the field $K$ (which is typically assumed to be algebraically closed or, at least, perfect).  Choose $1 \leq i_1 \leq \cdots \leq i_h \leq m$ and $1 \leq j \leq \cdots \leq j_h \leq n$, and for each $1 \leq a \leq h$ define $L_a = \{x_{i,j}  \mid i \leq i_a, j \leq j_a\}$ as a submatrix of $X$ and $L = \cup_{a = 1}^h L_a$.  Fix $\underline{t} = (t_1, \ldots, t_h)\} \in \mathbb{N}^h$.  Let $I_{\underline{t}}(L)$ be the ideal generated by the $t_a$-minors of $L_a$ for $1 \leq a \leq h$.   A one-sided \emph{mixed ladder determinantal variety} is the variety determined by $I_{\underline{t}}(L)$.
\end{definition}

Notice that our convention in this paper will be that every $L_a$ contains $x_{1,1}$ and is determined by its southeast corner $x_{i_a,j_a}$.  Elsewhere in the literature, particularly among those working with anti-diagonal term orders, the convention is that all $L_a$ include $x_{m,1}$ and are determined by their northeast corners.  Mixed ladder deterinantal varieties have also been studied in the case of two-sided ladders, as in \cite{Gor07}.  

Liaison theory provides a set of tools that has been used to study mixed ladder determinantal varieties.  We refer the reader to \cite{MN02} for a thorough overview and give a brief sketch below of some of the basic ideas as they relate to the approach in this paper.  Liaison theory studies unions of schemes in projective space.  We ask which desirable properties of the scheme $C_1$ can be inferred about $C_2$ if we know that $X = C_1 \cup C_2$ is well behaved.  Most relevant to this paper, if $X$ is Gorenstein and $C_1$ is Cohen-Macaulay, then $C_2$ must be Cohen-Macaulay as well.  More precisely,

\begin{definition}
Let $C_1, C_2, X \subseteq \mathbb{P}^n$ be subschemes defined by $I_{C_1}$, $I_{C_2}$, and $I_{X}$, respectively with $X$ arithmetically Gorenstein.  If $I_X \subseteq I_{C_1} \cap I_{C_2}$ and if $[I_X:I_{C_1}] = I_{C_2}$ and $[I_X: I_{C_2}] = I_{C_1}$, then $C_1$ and $C_2$ are \emph{directly algebraically $G$-linked}.
\end{definition}

Here the $G$ in $G$-linked stands for Gorenstein.  It contrasts with $CI$-linkage, in which we would insist that $X$ be a complete intersection.  We may generate an equivalence relation by these direct links.  One way to establish that two schemes are in the same $G$-liaison class is by giving an \emph{elementary $G$-biliaison} on their defining ideals.

\begin{definition}
Let $I$ and $J$ be homogeneous, saturated, height unmixed ideals in a standard graded polynomial ring $R$ with \(\height(I) = \height(J)\).  Then $J$ is obtained by an \emph{elementary G-biliaison} if there exists $\ell \in \mathbb{Z} $ and a homogeneous Cohen-Macaulay ideal $N \subseteq I \cap J$ with $(R/N)_P$ Gorenstein for every minimal prime $P$ of $N$ and with \(\height(N) = \height(J)-1\) and $J/N \cong [I/N](-\ell)$ as graded $R/N$-modules.  
\end{definition}

We refer the reader to \cite{Har07} and \cite{KMM+01} for a treatment of biliaison.  One particularly useful type of $G$-biliaison is a \emph{basic double $G$-link}.

\begin{definition}
Let $A \subseteq B$ be homogeneous ideals in the standard graded polynomial ring $R$ with $\height(A) = \height(B)-1$, $B$ height unmixed, and $A$ Cohen-Macaulay with $(R/A)_P$ Gorenstein for every minimal prime $P$ of $A$.  Let $f$ be a homogeneous element of $R$ that is not a zerodivisor on $R/A$, and define $C = A+fB$.  Then $C$ is a \emph{basic double $G$-link} of $B$ on $A$.  
\end{definition}

An object of much study is the Gorenstein liaison class of a complete intersection.  

\begin{definition} \label{glicci}
We say that a subscheme $C_1 \subseteq \mathbb{P}^n$ is \emph{glicci} if it can be obtained from some complete intersection $C_2$ by finitely many direct algebraic $G$-links.  
\end{definition}

It is known that a subscheme that is glicci is also arithmetically Cohen-Macaulay.  The converse is conjectured but not known.  Our conclusion that double determinantal ideals are, indeed, glicci adds to the literature of evidence in support of this conjecture.  Through the work of \cite{GM00}, \cite{Gor07}, and \cite{MN02}, among others, we see that questions about ideals generated by determinants are particularly amenable to inductive arguments using liaison theory.  Moreover, the algebra of the induction can often be reflected by a vertex decomposition of the associated Stanley-Reisner complexes (see \cite[Section 1.1]{MS04} for background).  In this paper, we use first the structure of this style of induction without direct appeal to liaison theory (Section \ref{specialsection}) and then make use of the full language of liaison theory (Section \ref{generalsection}) to solve a new problem arising from combinatorics, geometry, and representation theory.

\section{The special case of two matrices and maximal minors}
\label{specialsection}
In this section, we will use a style of inductive argument common in liaison theory as it applies to questions that are determinantal in nature to prove a special case of the main result of this paper.  The proof of the main result, which appears in Section \ref{generalsection}, is at its heart similar to this approach but is substantially more subtle.  We begin by recalling a key lemma \cite[Lemma 1.12]{GMN13} due to Gorla, Migliore, and Nagel.  Although they phrase the lemma in terms of liaison theory, that language is not essential.  We present the lemma below with only the necessary hypotheses and with a detailed proof suitable for broad audience.

For a $\mathbb{Z}$-graded module $M$ over a $\mathbb{Z}$-graded algebra $R$ over the field $K$, we will use $[M]_d$ to mean the degree $d$ summand of $M$.  For a homogeneous ideal $I$, we use $H_I(d)$ to denote the Hilbert function of $I$ evaluated at $d$, i.e., the dimension of $[R/I]_d$ as a $K$-vector space.  We use $\init(I)$ to denote the initial ideal of $I$ with respect to a specified term order.

\begin{lemma}\cite[Lemma 1.12]{GMN13} \label{givesGrobner}
Let $I$, $J$, and $N$ be homogeneous ideals of the graded algebra $R$ over the field $K$ with $N \subseteq I \cap J$.  Let $A$, $B$, and $C$ be homogeneous ideals of $R$ such that, with respect to some term order $\sigma$, $A \subseteq C \subseteq \init(J),$ $A = \init(N)$, and $B = \init(I)$.  Suppose that there exists $\ell \in \mathbb{Z}$ such that $[J/N]_d \cong [I/N]_{d-\ell}$ and that $[C/A]_d \cong [B/A]_{d-\ell}$ for all $d \in \mathbb{Z}$.  Then $C = \init(J)$.  
\end{lemma}

\begin{proof}
Because $C \subseteq \init(J)$, it is enough to show that $H_C(d) = H_J(d)$ for all $d \in \mathbb{Z}$.  We compute 

\begin{align*}
H_J(d)
&=  H_N(d)-\dim_K([J/N]_d) && (0 \rightarrow J/N \rightarrow R/N \rightarrow R/J \rightarrow 0)\\
&=  H_N(d)-\dim_K([I/N]_{d-\ell}) && ([J/N]_d \cong [I/N]_{d-\ell})\\
&= H_N(d)-H_N(d-\ell)+H_I(d-\ell) && (0 \rightarrow I/N \rightarrow R/N \rightarrow R/I \rightarrow 0)\\
& =H_A(d)-H_A(d-\ell)+H_B(d-\ell) && (A = \init(N) \mbox{ and } B = \init(I))\\
& = H_A(d) - \dim_K([B/A]_{d-\ell}) && (0 \rightarrow B/A \rightarrow R/A \rightarrow R/B \rightarrow 0)\\
& = H_A(d) - \dim_K([C/A]_d) && ([C/A]_d \cong [B/A]_{d-\ell})\\
& = H_C(d) && (0 \rightarrow C/A \rightarrow R/A \rightarrow R/C \rightarrow 0).
\end{align*}

\end{proof}

Although it is only formally required that $C$ be homogeneous, in applications of this lemma, $C$ will be already known to be monomial.  It is automatic that $A$ and $B$ are homogeneous since they are assumed to be monomial ideals.  The main use of this lemma is to show that the initial terms of the conjectured Gr\"obner basis for $J$ indeed generate the initial ideal of $J$.  

If $X$ and $Y$ are two $m \times n$ matrices, then we will let $H = (XY)$ be their $m \times 2n$ horizontal concatenation and $V = \begin{pmatrix} X\\ Y \end{pmatrix}$ be their $2m \times n$ vertical concatenation.  For a matrix $Z$ and integer $t$, let $I_t(Z)$ denote the ideal generated by the $t$-minors of $Z$.  

\begin{theorem}\label{maximalCase}
Let $X = (x_{i,j})$ and $Y = (y_{i,j})$ be $m \times n$ matrices of distinct indeterminates, and let $R = K[x_{i,j}, y_{i,j} \mid i \in [m], j \in [n]]$ be the standard graded polynomial ring in $2mn$ variables over a field $K$.  Let $J = I_m(H)+I_n(V)$ be an ideal of $R$.  Then the natural generators of $J$ form a Gr\"obner basis under any diagonal term order.
\end{theorem}

Before proving this theorem, we notice that the proposed Gr\"obner basis is never reduced since, for example, $x_{1,1} \cdots x_{m-1,m-1}y_{m,n}$ is the leading term of an $m$-minor of $H$ and divides the leading term of at least one $n$-minor of $V$ whenever $m \leq n$ and, even when $m=n$, those two minors themselves are not the same.
\begin{proof}
Without loss of generality, assume that $m \leq n$.  In order to set up an induction, we will need to introduce a broader set of rectangular matrices.  In particular, let $x_{i,j}$ and $y_{i,j}$ with $1 \leq i \leq m+k$, $1 \leq j \leq n+\ell$ be distinct indeterminates with $0 \leq k \leq \ell \leq k+m-2$, $0 \leq a \leq n$, and $-m+2 \leq b \leq 1, m-1$.  Suppose further that if $a = 0$, then $k = \ell$ and $b = 0$ and that if $a>0$, then $b = 1-(\ell-k)$.  If $\sigma$ is a diagonal term order, then we claim that the ideal generated by the $m$-minors of

  \[ 
    H^{m,n}_{k,a} = \begin{pmatrix}
      x_{1,1} & \ldots & x_{1,n+k} & y_{1,1} & \ldots & y_{1,n-a}\\
      \vdots &  & \vdots & \vdots & & \vdots \\
      x_{m,1} & \ldots & x_{m,n+k} & y_{m, 1} & \ldots & y_{m, n-a}\\
    \end{pmatrix}
  \]

\noindent  together with the $n$-minors of

  \[ 
    V^{m,n}_{\ell,b} = \begin{pmatrix}
      x_{1,1} & \ldots & x_{1,n}\\
      \vdots &  & \vdots  \\
      x_{m+\ell,1} & \ldots & x_{m+\ell,n} \\
      y_{1,1} & \ldots & y_{1,n}\\
      \vdots &  & \vdots  \\
      y_{m-b, 1} & \ldots & y_{m-b, n}\\
    \end{pmatrix}
  \]

\noindent in the polynomial ring in the variables appearing in at least one of $H^{m,n}_{k,a}$ or $V^{m,n}_{\ell,b}$ over the field $K$ has a Gr\"obner basis given by the natural generators.    We may assume that $n \leq 2m+\ell-b$ because otherwise the ideal in question is simply $I_m(H^{m,n}_{k,a})$, which is determinantal, and so the result is already known.

  Fix a diagonal term order $\sigma$ and $m$, $n$, $a$, $b$, $k$, and $\ell$ as above.  We will proceed by induction on $n-a$ (with base case $n-a=0$) and on $m$ (with base case $m=1$).  Notice that whenever $a = n$, the ideal in question is that of a mixed ladder determinantal variety, for which the conclusion is known (by \cite[Theorem 4.5.4]{GM00}, \cite[Theorem 1.10]{Gor07}, or \cite[Theorem 5.5]{GMN13}). \footnote{Although Gonciulea and Miller state their theorem for a particular anti-diagonal term-order, their argument is valid for any anti-diagonal term order, in which generality Gorla et al. state their results.  Notice that their one-sided ladders have a unique southwest corner and that ours have a unique northwest corner.  By reflecting our matrices over a horizontal line, one may transition between their ladders and an anti-diagonal term order and ours with a diagonal term order.}  Separately, if $m = 1$, then the ideal in question is given by the variables $x_{1,j}$ with $1 \leq j \leq n+k$ and $y_{1,j}$ with $1 \leq j \leq n-a$ together with the $n$-minors of
  
  \[ 
    \begin{pmatrix}
      x_{2,1} & \ldots & x_{2,n}\\
      \vdots & & \vdots  \\
      x_{1+\ell,1} & \ldots & x_{1+\ell,n} \\
      y_{1,1} & \ldots & y_{1,n} \\
        \vdots & & \vdots  \\
      y_{1-b,1} & \ldots & y_{1-b,n} \\
    \end{pmatrix}.
  \]
\noindent If $b<a$, every $y_{1,j}$ with $1 \leq j \leq n+k$ sits strictly below the main diagonal of any $n \times n$ submatrix of the above matrix whose $n$-minors we are examining, and so the result follows from the ordinary determinantal case together with \cite[Lemma 4.5.3]{GM00}.  If $b=a = 0$, then the ideal in question is simply a cone over the ideal of $n$-minors of 
  \[ 
    \begin{pmatrix}
      x_{2,1} & \ldots & x_{2,n}\\
      \vdots & & \vdots  \\
      x_{1+\ell,1} & \ldots & x_{1+\ell,n} \\
    \end{pmatrix}.
  \]
  
  \noindent Notice that $b = a = 1$ does not occur when $m =1$ because we have insisted $b \leq m-1$.
 
We now assume $2 \leq m$ and $a<n$.  Informally, the induction will proceed by relating the ideal of interest, $J$, to two other ideals, $I$ and $N$, whose construction will involve decreasing $m$ and increasing $a$, respectively.  In order to construct $I$, we will consider minors of size one smaller than those used to construct $J$ strictly northwest of $y_{m,n-a}$ in whichever of $H^{m,n}_{k,a}$ and $V^{m,n}_{\ell,b}$ contain $y_{m,n-a}$ as its southeast entry (for which the options are only $H^{m,n}_{k,a}$ if $a>0$ or both $H^{m,n}_{k,a}$ and $V^{m,n}_{\ell,b}$ if $a = 0$).  This construction will involve the increase of $\ell$ or $k$ as a way of bookkeeping a decrease of $m$ or $n$.  In particular, no new variables ever appear in any stage of this induction.  In the construction of $N$, we are deleting only the final column from $H^{m,n}_{k,a}$ in the case $0<a$ or both the final column from $H^{m,n}_{k,a}$ and also the final row of $V^{m,n}_{\ell,b}$ in the case $a = 0$.  In this construction, the sizes of the minors we are considering do not change.  
  
 \begin{notation}
Let $I_{k,\ell,a,b}(m,n) = I_m(H^{m,n}_{k,a})+I_n(V^{m,n}_{\ell,b})$.
 \end{notation} 
 
First consider the case $a = 0$.  Set $J = I_{k,\ell, 0,0}(m,n)$, $N = I_{k,\ell,1,1}(m,n)$, and $I = I_{k+1,\ell+1,0,0}(m-1,n-1)$.  Informally, $N$ is the ideal generated by the $m$-minors in the horizontal concatenation not involving its last column and the $n$-minors in the vertical concatenation not involving its last row, and $I$ is the ideal generated by the $(m-1)$-minors using only entries from $H^{m,n}_{k,0}$ that do not use the last row or column together with the $(n-1)$-minors of $V^{m,n}_{k,0}$ that do not use the last row or column.  Notice that each submatrix determining such a minor can be extended to a submatrix whose determinant is a generator of $J$ by augmenting it with the column and row containing $y_{m,n}$.  It is here that we are using that the minors are size $m$ and $n$ since those are the conditions that guarantee that a minor involving the last column of $H^{m,n}_{k,0}$ or last row of $V^{m,n}_{k,0}$ has identically $y_{m,n}$ as its southeast entry.  With respect to $\sigma$, set $A = \init(N)$, $B = \init(I)$, and let $C \subseteq \init(J)$ be the ideal generated by the main diagonals of the natural generators of $J$.  We must show that $C = \init(J)$.
 
We will first show that $[I/N](-1) \cong J/N$ as graded $R/N$-modules.  We first define an $R/N$-module map $\overline{\phi}: I/N \rightarrow J/N$.  Informally, we will be taking each generator of $I$ that is not also a generator of $N$ and mapping it to the generator of $J$ that includes it as a subdeterminant and has $y_{m,n}$ as the final entry of its leading term.

Fix $f \in I$, for which we know a Gr\"obner basis by induction.  Let $f_r$ denote its unique remainder after reduction by the Gr\"obner basis for $N$ given by the inductive hypothesis.  Using the deterministic division algorithm with respect to the order on the generators of $I$ induced by $\sigma$ within the degree $m-1$ generators coming from $H^{m,n}_{k,0}$ and then within the degree $n-1$ generators coming from $V^{m,n}_{\ell,0}$, there is a unique expression $f_r = \sum \alpha_i \delta_i$ of $f_r$, where the $\delta_i$ are the natural generators of $I$ that are not among the natural generators of $N$, the $\alpha_i \in R$, and $\alpha_i \init(\delta_i) \notin (\init(\delta_j) \mid \delta_j<_\sigma \delta_i)$.  If $\delta_i$ has degree $m-1$, then let $C(\delta_i)$ denote the determinant of the $m \times m$ submatrix of $H^{m,n}_{k,0}$ having $\delta_i$ as the minor if its $(m-1) \times (m-1)$ northwest submatrix and $y_{m,n}$ as its southeast corner.  Define $C(\delta_i)$ analogously in $V^{m,n}_{k,0}$ when $\delta_i$ has degree $n-1$ and is not in the ideal of $m-1$ minors of $H^{m,n}_{k,0}$.  We are now ready to define $\phi: I \rightarrow J/N$ by taking $\phi(f)$ to be the image in $J/N$ of $\sum \alpha_i C(\delta_i)$ where $f_r = \sum \alpha_i \delta_i$, as above.  Because we have used the deterministic division algorithm, $\phi$ is a well-defined map of sets, and it is not difficult to verity that $\phi$ is additive.  The challenge in verifying $R$-linearity is that for $s \in R$, $\sum s \alpha_i \delta_i$ is often \emph{not} the reduction of $sf_r$ obtained by the deterministic division algorithm.  Having established additivity, though, one immediately reduces to the case of $f_r = \alpha_p \delta_p$ for some $p$ and $s \alpha_p$ monomial.  Suppose first that $\delta_p$ is an $m$-minor in $H^{m,n}_{k,0}$.  If there is some counterexample, we fix a degree in which a counterexample occurs and then, within that degree, take $sf_r$ to have minimal leading term among counterexamples and, within that set, take $\delta_p$ to have the greatest leading term.  If the reduction of $sf_r$ obtained by the deterministic division algorithm is $s f_r$ itself, then the result is immediate. Otherwise, there will be some $x_{i,j}$ or some $y_{i,j}$ dividing $s\alpha_p$ such that $x_{i,j} \delta_p$ or $y_{i,j} \delta_p$ is not its own reduction obtained by the deterministic division algorithm.  
We will refer to that $x_{i,j}$ or $y_{i,j}$ as $z_{i,j}$ and $(s\alpha_p)/z_{i,j}$ as $(s\alpha_p)'$.  Now $s \alpha_p \delta_p = (s\alpha_p)'\sum z_{i,\ell} \delta_\ell$, where $\ell$ ranges over the columns of $\delta_p$ and each $\delta_\ell$ is obtained from $\delta_p$ by removing its column in position $\ell$ and replacing it with the column containing $z_{i,j}$ (which is either column $j$ or column $n+j$).  Now $s \alpha_p C(\delta_p) -(s\alpha_p)' \sum z_{i,\ell} C(\delta_\ell) = \pm y_{i,n} \mathcal{N} \in N$ where $\mathcal{N}$ is the $m$-minor given by the rows and columns of $\delta_p$ together with row $m$ and the column of $s$, which is an element of $N$.  Because each $(s\alpha_p)' z_{i,\ell} C(\delta_\ell)$ has either smaller leading term than $s \alpha_p \delta_p$ or equal leading term with the leading term of $\delta_\ell$ greater than the leading term of $\delta_p$, the result now follows from our assumptions of minimality. 
Suppose now that $\delta_p$ in an $n$-minor of $V^{m,n}_{k,0}$.  If the deterministic division algorithm applied to $s\alpha_p\delta_p$ makes use of any $m$-minors, then we obtain the desired result because each term in the resulting expression will involve only multiples of $(m-1)$-minors, which we understand by the argument above, or multiples of $(n-1)$-minors $\alpha_i \delta_i$ with the property that each term of $\alpha_i \delta_i$ is smaller with respect to $\sigma$ than the leading term of $s \alpha_p \delta_p$, which we understand by minimality of our choice of $s \delta_p$.  It, therefore, suffices to consider $s \delta_p$ whose reduction by the deterministic division algorithm uses only the generators of $V^{m,n}_{k,0}$, and there we repeat the argument from the previous case.  Therefore, $\phi$ is $R$-linear.  

It is clear that $\phi(N) = 0$, and so $\phi$ induces a map $\bar{\phi}:I/N \rightarrow J/N$.  If there is some $f \in I$ with leading term of $\sum \alpha_i C(\delta_i) \in J$ divisible by a main diagonal of any element of $N$, then $f_r = \sum \alpha_i \delta_i$ also has a term divisible by that same main diagonal because the leading term of $\sum \alpha_i C(\delta_i)$ is the product of $y_{m,n}$ and some term of $f_r$ and because $y_{m,n}$ is a non-zero-divisor on $R/\init(N)$.  But that is impossible by our choice of $f_r$, and so $\bar{\phi}$ is injective.  To show that $\bar{\phi}$ is surjective, it is sufficient to show that the image of $\phi$ includes the class of every generator of $J$ involving $y_{m,n}$.  Every $m$-generator in $H^{m,n}_{k,0}$ involving $y_{m,n}$ is the image of its cofactor corresponding to $y_{m,n}$ in a cofactor expansion.  Notice that any $n$-minor in $H^{m,n}_{k,0}$ having some term divisible by an $m$-minor of $H^{m,n}_{k,0}$ has, in particular, its leading term divisible by such an $m$-minor.  Therefore, the image of $\phi$ also includes all $n$-minors whose leading term is not divisible by an $m$-minor of $H^{m,n}_{k,0}$.  Using the deterministic division algorithm and a minimality argument, we have that every other $n$-minor is in the ideal of those we have already determined are in the image of $\phi$, and so $\phi$ and $\bar{\phi}$ are surjective.  It is clear that $\phi$ increases degree by exactly one, which concludes the proof that $[I/N](-1) \cong J/N$.

Similarly, define $\psi: [B/A](-1) \rightarrow C/A$ as a map of $R/A$-modules by multiplication by the image of $y_{m,n}$.   To see injectivity, we use here that $A$ is a monomial ideal generated by elements none of which are divisible by $y_{m,n}$.  We take surjectivity and well-definedness to be clear.  It now follows from Lemma \ref{givesGrobner} that $C = \init(J)$, as desired. 

The case of $0<a$ is very similar and in some ways simpler.  Here we will take $J = I_{k,\ell,a,b}(m,n)$, $I =I_{k,\ell+1,a+1,b-1}(m-1,n)$, and $N = I_{k,\ell, a+1,b}(m,n)$.  Now because $I_n(V_{\ell+1,b-1}^{m-1,n}) = I_n(V_{\ell,b}^{m,n}) \subseteq N$, when constructing the map $\bar{\phi}:[I/N](-1) \rightarrow J/N$, we have already that $f_r$, the remainder of $f$ after reduction by $N$, is an element of $I_{m-1}(H_{k,a+1}^{m-1,n})$, and so we may express $f_r = \sum \alpha_i \delta_i$ with the $\alpha_i \in R$ and the $\delta_i$ all degree $m-1$ natural generators of $H_{k,a+1}^{m-1,n}$.  We define $C(\delta_i)$ and then $\phi(f)$ as in the case $a = 0$.  The same arguments show that $\bar{\phi}$ is an isomorphism.  Again with $B = \init(I)$, $A = \init(N)$, and $C$ the ideal generated by the main diagonals of the natural generators of $J$, the map $\psi:[B/A](-1) \rightarrow C/A$ is given by multiplication by the image of $y_{m,n-a}$.  As before, Lemma \ref{givesGrobner} gives that $C = \init(J)$, concluding the proof.
\end{proof}

\begin{corollary}
With notation as in Theorem \ref{maximalCase}, $R/J$ is reduced and Cohen--Macaulay. 
\end{corollary}
\begin{proof}
By Theorem \ref{maximalCase}, the ideal $C$ generated by the main diagonals of the natural generators of $J$ under any diagonal term order $\sigma$ is the initial ideal of $J$.  It is clear that $C$ is square-free, and so $R/C$ is reduced, which implies that $R/J$ is reduced.  Similarly, it is sufficient to show that $R/C$ is Cohen-Macaulay.  We give two related proofs:  one combinatorial using the Stanley-Reisner complex and one purely algebraic.  

This argument is indifferent to case $a= 0$ and $a>0$ provided we interpret $I$ and $N$ to be defined together with respect to $J$ within one case or the other.  For this reason, we here suppress the distinction between the cases.  For notational convenience, we recall that $A = \init(N)$, $B = \init(I)$, and, by Theorem \ref{maximalCase}, $C = \init(J)$.

It is sufficient to show that the Stanley-Reisner complex of $R/C$ is vertex decomposable, which implies shellable.  Using throughout this claim the notation from Theorem \ref{maximalCase}, notice that $C = A+y_{m,n-a}B$ and that $(C:y_{m,n-a}R) = B$.  This is to say that lk$_{y_{m,n-a}}(\Delta_C) = \Delta_B$ and that $\Delta_C-\{y_{m,n-a}\} = \Delta_A$, where, $\Delta_A$ denotes the Stanley-Reisner complex of $R/A$, $\Delta_B$ the Stanley-Reisner complex of $R/B$, $\Delta_C$ the Stanley-Reisner complex of $R/C$, and lk denotes link.  Informally, these simplicial complexes are recording the information that exactly the main diagonals of generators of $I$ can be multiplied by $y_{m,n-a}$ to obtain a main diagonal or multiple of a main diagonal of a generator of $J$ and that the generators of $J$ that do not involve $y_{m,n-a}$ are the generators of $I$.  Because $y_{m,n-a}$ is a non-zero-divisor that generates the image of $C$ in $R/A$, we have that $\dim(R/C) = \dim(R/A)-1$.  Because $C \subseteq B$ with $y_{m,n-a}B \subseteq C$, $B$ is contained in the union of the associated primes of $C$, which are the same as the minimal primes of $C$ because $C$ is radical, and so $\dim(B) = \dim(C)$.  Because a Stanley-Reisner complex has dimension exactly one lower than the Krull dimension of its Stanley-Reisner ring, we have $\dim(\Delta_B) = \dim(\Delta_C) = \dim(\Delta_A)-1$.  By induction, $\Delta_B$ and $\Delta_A$ are vertex decomposable, and so by definition of vertex decomposition $\Delta_C$ is as well.  

We will also proceed by induction to give a purely algebraic proof.  Using that $(C:y_{m,n-a}R) = B$, we have the four-term exact sequence

\[
  0 \rightarrow B(R/C) \rightarrow R/C \xrightarrow{y_{m,n-a}} R/C \rightarrow R/(C+y_{m,n-a}R) \rightarrow 0,
\]
which gives rise to the short exact sequence
\[
  0 \rightarrow R/B \xrightarrow{y_{m,n-a}} R/C \rightarrow R/(A+y_{m,n-a}R) \rightarrow 0
\]
using the observation that $R/(A+y_{m,n-a}R) \cong R/(C+y_{m,n-a}R)$.  Because $y_{m,n-a}$ is a non-zero-divisor on $A$, we have $\dim(R/(A+y_{m,n-a}R)) = \dim(R/A)-1$, and so, in particular, recalling our dimension arguments from the preceding paragraph, $\dim(R/C) = \dim(R/B) = \dim(R/(A+y_{m,n-a}R))$.  By induction $R/B$ and $R/A$ are Cohen-Macaulay, and the quotient of a Cohen-Macaulay ring by a non-zero-divisor is always again Cohen-Macaulay, and so $R/(A+y_{m,n-a}R)$ is Cohen-Macaulay.  Now because $R/(A+y_{m,n-a}R)$ and $R/B$ have only one non-vanishing $\Ext^\bullet(R/m,\functorslot)$ or Koszul cohomology or local cohomology on the homogeneous maximal ideal module, the long exact sequence for any of those shows that $R/C$ has only one non-vanishing module of the same, which is to say that it is Cohen-Macaulay.  

\end{proof} 

\section{The general case}
\label{generalsection}
\begin{theorem}\label{generalCase}
The natural generators of a double determinantal ideal form a Gr\"obner basis under any diagonal term order.  Double determinantal varieties are reduced and arithmetically Cohen-Macaulay.  In addition, they are glicci.
\end{theorem}

Before the proof begins, we outline the argument. Let $V_r$ be the double determinantal variety associated to the $m \times n$ matrices $X_1, \ldots, X_{r-1}$, and $Y$ of distinct indeterminates, where $X_q = (x^q_{i,j})$ and $Y = (y_{i,j})$ with minor sizes $s \leq m$ in the horizontal concatenation and $t \leq n$ in the vertical concatenation.  Let $R = K[X,Y] = K[x^q_{i,j}, y_{i,j} \mid i \in [m], j \in [n], q \in [r-1]]$ be the standard graded polynomial ring associated to the $X_q$ and $Y$ over the field $K$.  Fix a diagonal term order $\sigma$.  We must show that the ideal generated by the $s$-minors of the horizontal concatenation of the $X_q$ and $Y$ together with the $t$-minors of the vertical concatenation is Cohen-Macaulay and that the natural generators form a Gr\"obner basis.  In order to facilitate an induction, we will prove more generally that for certain ideals generated by some $s$-minors and some minors of various smaller sizes in restricted regions of the the horizontal concatenation together with some $t$-minors and some smaller minors in the vertical concatenation, the natural generating set forms a Gr\"obner basis.  In particular, we will be considering sums of two ideals of possibly mixed minors of particular one-sided ladders.  We will use Young diagrams (oriented in the English style) to track our argument.  Our induction will be foremost on $r$ and, within any fixed $r$, on the integer of which the largest Young diagram appearing is a partition.

\begin{proof}
  
Our proof will make use of the following assumptions and notation. Without loss of generality, assume $s \leq t$.  Fix $k \leq s$ Young diagrams $\lambda^0 = (\lambda^0_1, \ldots, \lambda^0_{\ell_0}), \cdots, \lambda^{k-1} = (\lambda^{k-1}_1, \ldots, \lambda^{k-1}_{\ell_{k-1}})$ with $\lambda^i \subseteq \lambda^{i-1}$ for all $0 \leq i < k$, $(r-1)n \leq \lambda^i_j \leq rn$ for each $0 \leq i < k$ and each $1 \leq j \leq \ell_i$.  Insist further that whenever $(i,(r-1)n+j)$ is a southeast corner of $\lambda^\alpha$ for any $0 <\alpha <k$, then $(i+1,(r-1)n+j+1) \notin \lambda^{\alpha-1}$.  Informally, these quantifiers are saying that each Young diagram fits inside of the horizontal concatenation of our $r$ matrices of indeterminates and covers the variables $x^q_{i,j}$ for $1 \leq q <r$ whenever any part of the Young diagram reaches the $i^{th}$ row and also that an outside corner of a smaller Young diagram must lie along the outer boundary of any Young diagram that contains it.  Eventually, our induction will shrink the Young diagrams so that the $y_{i,j}$ are in use only in the case of size-one minors, which will allow us to reduce $r$ by $1$.  Let $H_{\Lambda}$ denote the horizontal concatenation of the $X_q$ and $Y$ with $\lambda^0, \ldots, \lambda^{k-1}$ laid over it and justified northwest.  From each $\lambda^i$, create the Young diagram $\widehat{\lambda}^i = (\widehat{\lambda}^i_1, \ldots, \widehat{\lambda}^i_{m+k})$ with $\widehat{\lambda}^i_j = \lambda^i_1-(r-1)n$ for $1 \leq j \leq m$ and $\widehat{\lambda}^i_j = \lambda^i_{j-m}-(r-1)n$ for $m+1 \leq j \leq m+\ell_i$.  

Let $V_{\Lambda}$ denote the vertical concatenation of the $X_q$ and $Y$ with $\widehat{\lambda}_i$ for $0 \leq i < k$ laid over it and justified northwest.  Notice that each $\lambda^i$ and $\widehat{\lambda}^i$ cover the same variables $y_{i,j}$ but typically different $x^q_{i,j}$.  Notice also that $\widehat{\lambda}^i$ is entirely determined by $\lambda^i$, but that two distinct $\lambda^i$ can define the same $\widehat{\lambda}^i$, as we will see in Example \ref{inductiveExample}.  

Define the ideal $I_{\Lambda}(s,t)$ to be the ideal generated by the $(s-i)$-minors of submatrices of $H_{\Lambda}$ all of whose entries lie in $\lambda^i$ together with the $(t-i)$-minors of submatrices of $V_{\Lambda}$ all of whose entries lie in $\widehat{\lambda}^i$ for $0 \leq i<k$.  

We will sometimes say $y_{i,j} \in \lambda^\alpha$ to mean $(i,(r-1)n+j) \in \lambda^\alpha$ for $0 \leq \alpha < k$.  

\begin{claim}\label{internalClaim}
With notation as above, for any diagonal term order $\sigma$, the natural generators of $I_{\Lambda}(s,t)$ form a Gr\"obner basis for the ideal they generate.
\end{claim}

Theorem \ref{generalCase} is the case $k = 1$ and $\lambda^0_j = rn$ for all $1 \leq j \leq m$.  Recall that $\init(I_{\Lambda}(s,t))$ being glicci implies that $\init(I_{\Lambda}(s,t))$ is Cohen-Macaulay, which in turn implies that $I_{\Lambda}(s,t)$ is Cohen-Macaulay, and that $\init(I_{\Lambda}(s,t))$ square-free implies that $\init(I_{\Lambda}(s,t))$ is radical, which in turn implies that $I_{\Lambda}(s,t)$ is radical.  We will proceed by induction first on $r \geq 1$ and then, within each choice of $r$, on the number $\lambda^0$ partitions, i.e., the number of variables covered by $\lambda^0$.  

Our outermost induction is on $r$, the number of $m \times n$ matrices of indeterminates. For this induction, the base case is when $r = 1$ and $\lambda^i_j = n$ for for all all $1 \leq j \leq \ell_k$ for all $1 \leq j \leq k$, in which case  $I_{\Lambda}(s,t)$ is a mixed ladder determinantal ideal, and so the conclusion is known \cite[Corollary 2.2]{Gor07}\cite[Theorem 4.4.1]{GM00}.  Informally, this situation describes the case of one full matrix $X_1$ and minor sizes varying only in rectangular regions described as those north of some row with minor sizes decreasing as we consider progressively more northern submatrices.  

Having fixed an $r \geq 2$, if $\lambda^0_1 = (r-1)n$ then no Young diagram involves any variable $y_{i,j}$, and so we may view $I_{\Lambda}(s,t)$ as defined using only the $r -1$ matrices $X_q$ for $1 \leq q \leq r-1$. Also, if $k = s-1$, and if there exists $h \geq 1$ such that $\lambda^0_j = \lambda^{k}_j$ for all $j \leq h$ and either $\lambda^i_h = (r-1)n$ or $h = \ell_0$, then we may also view $I_{\Lambda}(s,t)$ as defined using only the $r -1$ matrices $X_q$ for $1 \leq q \leq r-1$.    These conditions describe the case of taking $1$-minors in any Young diagram involving any $y_{i,j}$, in which case $I_{\Lambda}(s,t)$ defines a cone over a double determinantal variety defined only using the matrices $X_q$. Together, these conditions describe the base cases for the induction on the size of the partition \(\lambda_{0}\).

With these base cases in mind, we proceed to the inductive argument. We assume that $s > 1$ and that there exists at least one $y_{i,j} \in \lambda^0$.  If there exists some $\alpha>0$ such that there exists some $y_{i,j} \in \lambda^0 - \lambda^\alpha$, choose $\alpha$ minimal with respect to that condition and let $y_{i,j}$ be the southernmost entry from the easternmost column of $\lambda^0 - \lambda^\alpha$.  If there does not exist such an $\alpha$, then choose $y_{i,j}$ to be simply the southernmost entry from the easternmost column of $\lambda^0$.  

We will proceed by removing $y_{i,j}$ in order to make $\lambda^0$ smaller.  Set $\widetilde{\lambda}^\beta = (\lambda^\beta_1, \ldots, \lambda^\beta_i - 1, \ldots, \lambda^\beta_{\ell_\beta})$, i.e., $\lambda^\beta$ with $y_{i,j}$ removed, for all $\lambda^\beta$ with $y_{i,j} \in \lambda^\beta$.  Define $I_{\overline{\Lambda}}(s,t)$ analogously to $I_\Lambda(s,t)$ with $\overline{\Lambda}$ referring to the set $\{\lambda^\beta \mid y_{i,j} \notin \lambda^\beta \} \cup \{\widetilde{\lambda^\beta} \mid y_{i,j} \in \lambda^\beta\}$.  If $I_{\overline{\Lambda}}(s,t) = I_{\Lambda}(s,t)$, then we are done by induction.  (This situation occurs when all of the $(s-\beta)$-minors in $H_{\Lambda}$ and all of the $(t-\beta)$-minors in $V_{\Lambda}$ that have $y_{i,j}$ as the final entry along their main diagonals are already in the ideal generated by smaller minors of submatrices northwest of $y_{i,j}$.)  Otherwise, if $\beta$ is the largest index for which $y_{i,j} \in \lambda^\beta$, we set $\widetilde{\lambda}^{\beta+1} = \lambda^{\beta+1} \cup \{r,s \mid r<i, s<(r-1)n+j\}$ if $\beta<k$, and set $\widetilde{\lambda}^{\beta+1} = \{r,s \mid r<i, s<(r-1)n+j\}$ if $\beta = k$.  Define $\widehat{\widetilde{\lambda^i}}$ analogously to $\widehat{\lambda}^i$ for each $1 \leq i \leq \beta+1$.  Less formally, we are adding to $\lambda^{\beta+1}$ and $\widehat{\lambda}^{\beta+1}$ all variables strictly northwest of $y_{i,j}$ so that every $(s-\beta)$-minor of a matrix with $y_{i,j}$ as its southeast entry in $\lambda^{\beta}$ is in the ideal generated by the $(s-(\beta+1))$-minors of the submatrices all of whose entries are in $\widetilde{\lambda}^{\beta+1}$ and the same for $(t-\beta)$-minors and $\widehat{\widetilde{\lambda^\beta}}$.  Let $\widetilde{\Lambda}$ denote the set $\{\widetilde{\lambda}^i \mid 1 \leq i \leq \beta+1 \} \cup \{\lambda^i \mid \beta+1<i \leq k\}$.  

We now prepare to use Lemma \ref{givesGrobner}.  Set $N = I_{\overline{\Lambda}}(s,t)$, $J = I_{\Lambda}(s,t)$, and $I = I_{\widetilde{\Lambda}}(s,t)$.  Notice that $N \subseteq I \cap J$ and that, by induction, the natural generators of $I$ and $N$ form a Gr\"obner basis for $I$ and $N$, respectively.  Set $A = \init(N)$ and $B = \init(I)$, and let $C$ be the ideal generated by the main diagonals of the natural generators of $J$.  We will exhibit an isomorphism $\phi: [I/N](-1) \rightarrow J/N$.  

In a manner similar to the argument in Section 3, define $f_r$ to be the remainder of $f$ after reduction by $N$, for which we know a Gr\"obner basis by induction.  Let $\delta_u$ be the natural generators of $I$.  Using the deterministic division algorithm, divide first by the $\delta_u$ coming from $H_{\Lambda}$ in the order induced by $\sigma$ and by then those coming from $V_{\Lambda}$ ordered by $\sigma$ to obtain $f_r = a_u \delta_u$.  For any generator $\delta_u$ in the expression coming from $H_{\Lambda}$, let $C(\delta_u)$ denote the determinant of the submatrix of $H_{\Lambda}$ having $\delta_u$ as the determinant of its northwest corner and $y_{i,j}$ its southeast entry.  Define $C(\delta_u)$ analogously for any $\delta_u$ in the expression coming from $V_{\Lambda}$.  Define $\phi: [I/N](-1) \rightarrow J/N$ by $\phi(f) = \sum a_u C(\delta_u)$.  We leave the details of the verification that $\phi$ is an isomorphism to the reader and point the reader to the similar argument in Section 3 as a guide.  Furthermore, the map $\psi:[B/A](-1) \xrightarrow{y_{i,j}} C/A$ is an isomorphism.  It follows from Lemma \ref{givesGrobner} that $C = \init(J)$, which completes the proof of Claim \ref{internalClaim}.

Because $A$, $B$, and $C$ are clearly square-free, they are radical, and so $N$, $I$, and $J$ are also radical. It follows that $A$ and $N$ are also $G_0$.  All of these ideals are saturated (having excluded the case $s=1$ and $\lambda^0 = H$) and homogeneous.  Because $A \subseteq C \subseteq B$, we have $\height(A) \leq \height(C) \leq \height(B)$.  But the image of $C$ is generated by $y_{i,j}$ in $R/A$ with $y_{i,j}$ a non-zero-divisor in $R/A$, so $\height(A) = \height(C)+1$.  Also we have $y_{i,j}B \subseteq C$, and so $\height(C) = \height(B)$.  Now $C = A+y_{i,j}B$, and so $C$ is a basic double $G$-link of $B$ on $A$, and so $A$ and $B$ Cohen--Macaulay (and so also height unmixed) by induction implies $C$ Cohen-Macaulay, which in turn implies $J$ Cohen-Macaulay.  Moreover, the map $\phi$ shows that $J$ is obtained from $I$ via elementary $G$-bilianson of height $1$.  Because $I$ is glicci by induction, $J$ is also glicci.   
\end{proof}

\begin{remark} Note that, as in Section 3 and following the notation introduced there, we have the relations lk$_{y_{i,j}}(\Delta_C) = \Delta_B$ and $\Delta_C-\{y_{i,j}\} = \Delta_A$ on Stanley-Reisner complexes, which shows by induction that $\Delta_C$ is vertex decomposable and so gives a combinatorial argument that $R/C$ is Cohen-Macaulay.  
\end{remark}

\begin{example}\label{inductiveExample}
  If $m=5$, $n=7$, $s=3$, $t=4$, $\lambda = \lambda^0= (12,12,11,11)$, $\mu = \lambda^1= (12,12)$, and $y_{i,j} = y_{4,4}$, the matrices together with Young diagrams and ideals that make up the inductive step of this argument are described explicitly below.  With only two $\lambda^i$ in the example, we will always write out $\Lambda$, $\overline{\Lambda}$, and $\widetilde{\Lambda}$ explicitly in terms of $\lambda$ and $\mu$.  In the following matrices, we use a solid border to indicate \(\lambda\) and a dotted line for \(\mu\). We begin with the \(5 \times 14\) matrix shown below, with \(\lambda\) and \(\mu\) overlaid.

  \begin{equation*}
    H_{\lambda,\mu} = \begin{tikzpicture}[baseline=(current  bounding  box.center)]
      \matrix [matrix of math nodes,left delimiter=(,right delimiter=)] (m)
      {
        x_{11} & x_{12} & x_{13} & x_{14} & x_{15} & x_{16} & x_{17} & y_{11} & y_{12} & y_{13} & y_{14} & y_{15} & y_{16} & y_{17}\\
        x_{21} & x_{22} & x_{23} & x_{24} & x_{25} & x_{26} & x_{27} & y_{21} & y_{22} & y_{23} & y_{24} & y_{25} & y_{26} & y_{27}\\
        x_{31} & x_{32} & x_{33} & x_{34} & x_{35} & x_{36} & x_{37} & y_{31} & y_{32} & y_{33} & y_{34} & y_{35} & y_{36} & y_{37}\\
        x_{41} & x_{42} & x_{43} & x_{44} & x_{45} & x_{46} & x_{47} & y_{41} & y_{42} & y_{43} & y_{44} & y_{45} & y_{46} & y_{47}\\
        x_{51} & x_{52} & x_{53} & x_{54} & x_{55} & x_{56} & x_{57} & y_{51} & y_{52} & y_{53} & y_{54} & y_{55} & y_{56} & y_{57}\\
      };
      \draw[black]
      ([xshift = -2pt, yshift = 2pt] m-1-1.north west) --
      ([xshift = -2pt, yshift = -2pt] m-4-1.south west) --
      ([xshift = 2pt, yshift = -2pt] m-4-11.south east) --
      ([xshift = 2pt, yshift = -2pt] m-2-11.south east) --
      ([xshift = 2pt, yshift = -2pt] m-2-12.south east) --
      ([xshift = 2pt, yshift = 2pt] m-1-12.north east) --
      ([xshift = -2pt, yshift = 2pt] m-1-1.north west);

      \draw[black,dotted]
      (m-1-1.north west) --
      (m-2-1.south west) --
      (m-2-12.south east) --
      (m-1-12.north east) --
      (m-1-1.north west);
    \end{tikzpicture}
  \end{equation*}
  
   We have \(\widehat{\lambda} = (5,5,5,5,5,5,5,4,4)\) and \(\widehat{\mu} = (5,5,5,5,5,5,5)\), shown below.  Notice that $\widehat{\lambda}$ and $\lambda$ (respectively, $\widehat{\mu}$ and $\mu$) cover exactly the same variables from the matrix $Y$ but some different variables from the matrix $X$. 
  
  \begin{equation*}
    V_{\lambda,\mu} = \begin{tikzpicture}[baseline=(current  bounding  box.center)]
      \matrix [matrix of math nodes,left delimiter=(,right delimiter=)] (m)
      {
        x_{11} & x_{12} & x_{13} & x_{14} & x_{15} & x_{16} & x_{17}\\ 
        x_{21} & x_{22} & x_{23} & x_{24} & x_{25} & x_{26} & x_{27}\\
        x_{31} & x_{32} & x_{33} & x_{34} & x_{35} & x_{36} & x_{37}\\
        x_{41} & x_{42} & x_{43} & x_{44} & x_{45} & x_{46} & x_{47}\\
        x_{51} & x_{52} & x_{53} & x_{54} & x_{55} & x_{56} & x_{57}\\
        y_{11} & y_{12} & y_{13} & y_{14} & y_{15} & y_{16} & y_{17}\\
        y_{21} & y_{22} & y_{23} & y_{24} & y_{25} & y_{26} & y_{27}\\
        y_{31} & y_{32} & y_{33} & y_{34} & y_{35} & y_{36} & y_{37}\\
        y_{41} & y_{42} & y_{43} & y_{44} & y_{45} & y_{46} & y_{47}\\
        y_{51} & y_{52} & y_{53} & y_{54} & y_{55} & y_{56} & y_{57}\\
      };

      \draw[black]
      ([xshift = -2pt, yshift = 2pt] m-1-1.north west) --
      ([xshift = -2pt, yshift = -2pt] m-9-1.south west) --
      ([xshift = 2pt, yshift = -2pt] m-9-4.south east) --
      ([xshift = 2pt, yshift = -2pt] m-7-4.south east) --
      ([xshift = 2pt, yshift = -2pt] m-7-5.south east) --
      ([xshift = 2pt, yshift = 2pt] m-1-5.north east) --
      ([xshift = -2pt, yshift = 2pt] m-1-1.north west);

      \draw[black,dotted]
      (m-1-1.north west) --
      (m-7-1.south west) --
      (m-7-5.south east) --
      (m-1-5.north east) --
      (m-1-1.north west);
    \end{tikzpicture}
  \end{equation*}

Here $J = I_{\lambda, \mu}(3,4)$ is the ideal generated by the $3$-minors inside of the solid line of $H_{\lambda, \mu}$, the $2$-minors inside of the dotted line of $H_{\lambda, \mu}$, the $4$-minors inside of the solid line of $V_{\lambda,\mu}$ and the $3$-minors inside of the dotted line of $V_{\lambda,\mu}$.  

The southernmost entry in the easternmost column of \(\lambda - \mu\) is \(y_{44}\). Removing this entry from \(\lambda\) gives \(\widetilde{\lambda} = (12, 12, 11, 10)\), $\widehat{\widetilde{\lambda}} = (5,5,5,5,5,5,4,3)$, and the following matrices.

  \begin{equation*}
    H_{\widetilde{\lambda},\mu} = \begin{tikzpicture}[baseline=(current  bounding  box.center)]
      \matrix [matrix of math nodes,left delimiter=(,right delimiter=)] (m)
      {
        x_{11} & x_{12} & x_{13} & x_{14} & x_{15} & x_{16} & x_{17} & y_{11} & y_{12} & y_{13} & y_{14} & y_{15} & y_{16} & y_{17}\\
        x_{21} & x_{22} & x_{23} & x_{24} & x_{25} & x_{26} & x_{27} & y_{21} & y_{22} & y_{23} & y_{24} & y_{25} & y_{26} & y_{27}\\
        x_{31} & x_{32} & x_{33} & x_{34} & x_{35} & x_{36} & x_{37} & y_{31} & y_{32} & y_{33} & y_{34} & y_{35} & y_{36} & y_{37}\\
        x_{41} & x_{42} & x_{43} & x_{44} & x_{45} & x_{46} & x_{47} & y_{41} & y_{42} & y_{43} & y_{44} & y_{45} & y_{46} & y_{47}\\
        x_{51} & x_{52} & x_{53} & x_{54} & x_{55} & x_{56} & x_{57} & y_{51} & y_{52} & y_{53} & y_{54} & y_{55} & y_{56} & y_{57}\\
      };

      \draw[black]
      ([xshift = -2pt, yshift=2pt] m-1-1.north west) --
      ([xshift = -2pt, yshift=-2pt] m-4-1.south west) --
      ([xshift = 2pt, yshift=-2pt] m-4-10.south east) --
      ([xshift = 2pt, yshift=-2pt] m-3-10.south east) --
      ([xshift = 2pt, yshift=-2pt] m-3-11.south east) --
      ([xshift = 2pt, yshift=-2pt] m-2-11.south east) --
      ([xshift = 2pt, yshift=-2pt] m-2-12.south east) --
      ([xshift = 2pt, yshift=2pt] m-1-12.north east) --
      ([xshift = -2pt, yshift=2pt] m-1-1.north west);

      \draw[black,dotted]
      (m-1-1.north west) --
      (m-2-1.south west) --
      (m-2-12.south east) --
      (m-1-12.north east) --
      (m-1-1.north west);
    \end{tikzpicture}
  \end{equation*}
  
    \begin{equation*}
    V_{\widetilde{\lambda},\mu} = \begin{tikzpicture}[baseline=(current  bounding  box.center)]
      \matrix [matrix of math nodes,left delimiter=(,right delimiter=)] (m)
      {
        x_{11} & x_{12} & x_{13} & x_{14} & x_{15} & x_{16} & x_{17}\\ 
        x_{21} & x_{22} & x_{23} & x_{24} & x_{25} & x_{26} & x_{27}\\
        x_{31} & x_{32} & x_{33} & x_{34} & x_{35} & x_{36} & x_{37}\\
        x_{41} & x_{42} & x_{43} & x_{44} & x_{45} & x_{46} & x_{47}\\
        x_{51} & x_{52} & x_{53} & x_{54} & x_{55} & x_{56} & x_{57}\\
        y_{11} & y_{12} & y_{13} & y_{14} & y_{15} & y_{16} & y_{17}\\
        y_{21} & y_{22} & y_{23} & y_{24} & y_{25} & y_{26} & y_{27}\\
        y_{31} & y_{32} & y_{33} & y_{34} & y_{35} & y_{36} & y_{37}\\
        y_{41} & y_{42} & y_{43} & y_{44} & y_{45} & y_{46} & y_{47}\\
        y_{51} & y_{52} & y_{53} & y_{54} & y_{55} & y_{56} & y_{57}\\
      };

      \draw[black]
      ([xshift = -2pt, yshift = 2pt] m-1-1.north west) --
      ([xshift = -2pt, yshift = -2pt] m-9-1.south west) --
      ([xshift = 2pt, yshift = -2pt] m-9-3.south east) --
      ([xshift = 2pt, yshift = -2pt] m-9-3.north east) --
      ([xshift = 2pt, yshift = -2pt] m-8-4.south east) --
      ([xshift = 2pt, yshift = -2pt] m-8-4.north east) --
      ([xshift = 2pt, yshift = -2pt] m-8-5.north east) --
      ([xshift = 2pt, yshift = 2pt] m-1-5.north east) --
      ([xshift = -2pt, yshift = 2pt] m-1-1.north west);

      \draw[black,dotted]
      (m-1-1.north west) --
      (m-7-1.south west) --
      (m-7-5.south east) --
      (m-1-5.north east) --
      (m-1-1.north west);
      
    \end{tikzpicture}
  \end{equation*}
  
  \bigskip
  
  The matrices $H_{\widetilde{\lambda},\mu}$ and $V_{\widetilde{\lambda},\mu}$ determine the ideal $N = I_{\widetilde{\lambda},\mu}(3,4)$, which is generated by the $3$-minors inside of the solid line of $H_{\widetilde{\lambda}, \mu}$, the $2$-minors inside of the dotted line of $H_{\widetilde{\lambda}, \mu}$, the $4$-minors inside of the solid line of $V_{\widetilde{\lambda},\mu}$ and the $3$-minors inside of the dotted line of $V_{\widetilde{\lambda},\mu}$.  Notice that there exist $3$-minors of $H_{\mu, \lambda}$ involving $y_{44}$, such as the minor of the submatrix whose entries along the main diagonal are $\{y_{22}, y_{33}, y_{44}\}$, that are not in the ideal $N = I_{\widetilde{\lambda},\mu}(3,4)$, and so in this case $N \subsetneq J$. 
  
To construct the matrices corresponding to $I = I_{\widetilde{\lambda},\widetilde{\mu}}(3,4)$, we add to $\mu$ every entry that is strictly northwest of the deleted \(y_{44}\).  We obtain \(\widetilde{\mu} = (12, 12, 10)\), shown below.

\bigskip
  
  \begin{equation*}
    H_{\widetilde{\lambda},\widetilde{\mu}} = \begin{tikzpicture}[baseline=(current  bounding  box.center)]
      \matrix [matrix of math nodes,left delimiter=(,right delimiter=)] (m)
      {
        x_{11} & x_{12} & x_{13} & x_{14} & x_{15} & x_{16} & x_{17} & y_{11} & y_{12} & y_{13} & y_{14} & y_{15} & y_{16} & y_{17}\\
        x_{21} & x_{22} & x_{23} & x_{24} & x_{25} & x_{26} & x_{27} & y_{21} & y_{22} & y_{23} & y_{24} & y_{25} & y_{26} & y_{27}\\
        x_{31} & x_{32} & x_{33} & x_{34} & x_{35} & x_{36} & x_{37} & y_{31} & y_{32} & y_{33} & y_{34} & y_{35} & y_{36} & y_{37}\\
        x_{41} & x_{42} & x_{43} & x_{44} & x_{45} & x_{46} & x_{47} & y_{41} & y_{42} & y_{43} & y_{44} & y_{45} & y_{46} & y_{47}\\
        x_{51} & x_{52} & x_{53} & x_{54} & x_{55} & x_{56} & x_{57} & y_{51} & y_{52} & y_{53} & y_{54} & y_{55} & y_{56} & y_{57}\\
      };
      
      \draw[black]
      ([xshift = -2pt,yshift = 2pt] m-1-1.north west) --
      ([xshift = -2pt,yshift = -2pt] m-4-1.south west) --
      ([xshift = 2pt,yshift = -2pt] m-4-10.south east) --
      ([xshift = 2pt,yshift = -2pt] m-3-10.south east) --
      ([xshift = 2pt,yshift = -2pt] m-3-11.south east) --
      ([xshift = 2pt,yshift = -2pt] m-2-11.south east) --
      ([xshift = 2pt,yshift = -2pt] m-2-12.south east) --
      ([xshift = 2pt,yshift = 2pt] m-1-12.north east) --
      ([xshift = -2pt,yshift = 2pt] m-1-1.north west);
      
      \draw[black,dotted]
      (m-1-1.north west) --
      (m-3-1.south west) --
      (m-3-11.south west) --
      (m-3-11.north west) --
      (m-3-12.north east) --
      (m-1-12.north east) --
      (m-1-1.north west);
    \end{tikzpicture}
  \end{equation*}
  
  \bigskip

And when we add all variables strictly northwest of $y_{44}$ to $\widehat{\mu}$, we get

\bigskip
  
    \begin{equation*}
    V_{\widetilde{\lambda},\widetilde{\mu}} = \begin{tikzpicture}[baseline=(current  bounding  box.center)]
      \matrix [matrix of math nodes,left delimiter=(,right delimiter=)] (m)
      {
        x_{11} & x_{12} & x_{13} & x_{14} & x_{15} & x_{16} & x_{17}\\ 
        x_{21} & x_{22} & x_{23} & x_{24} & x_{25} & x_{26} & x_{27}\\
        x_{31} & x_{32} & x_{33} & x_{34} & x_{35} & x_{36} & x_{37}\\
        x_{41} & x_{42} & x_{43} & x_{44} & x_{45} & x_{46} & x_{47}\\
        x_{51} & x_{52} & x_{53} & x_{54} & x_{55} & x_{56} & x_{57}\\
        y_{11} & y_{12} & y_{13} & y_{14} & y_{15} & y_{16} & y_{17}\\
        y_{21} & y_{22} & y_{23} & y_{24} & y_{25} & y_{26} & y_{27}\\
        y_{31} & y_{32} & y_{33} & y_{34} & y_{35} & y_{36} & y_{37}\\
        y_{41} & y_{42} & y_{43} & y_{44} & y_{45} & y_{46} & y_{47}\\
        y_{51} & y_{52} & y_{53} & y_{54} & y_{55} & y_{56} & y_{57}\\
      };
      \draw[black]
      ([xshift = -2pt, yshift = 2pt] m-1-1.north west) --
      ([xshift = -2pt, yshift = -2pt] m-9-1.south west) --
      ([xshift = 2pt, yshift = -2pt] m-9-3.south east) --
      ([xshift = 2pt, yshift = -2pt] m-9-3.north east) --
      ([xshift = 2pt, yshift = -2pt] m-8-4.south east) --
      ([xshift = 2pt, yshift = -2pt] m-8-4.north east) --
      ([xshift = 2pt, yshift = -2pt] m-8-5.north east) --
      ([xshift = 2pt, yshift = 2pt] m-1-5.north east) --
      ([xshift = -2pt, yshift = 2pt] m-1-1.north west);

      \draw[black,dotted]
      (m-1-1.north west) --
      (m-8-1.south west) --
      (m-8-3.south east) --
      (m-7-3.south east) --
      (m-7-5.south east) --
      (m-1-5.north east) --
      (m-1-1.north west);
    \end{tikzpicture}.
  \end{equation*}
\end{example}

Notice also that the stage of the induction that will call for the removal of $y_{3,4}$ from $I_{\widetilde{\lambda},\widetilde{\mu}}(3,4)$ will not actually change the defining ideal of the variety because every $3$-minor of $H_{\widetilde{\lambda}}$ involving $y_{3,4}$ is already in the ideal of $2$-minors of $H_{\widetilde{\mu}}$ and the same for the $4$-minors of $V_{\widetilde{\lambda}}$ involving $y_{3,4}$ and the $3$-minors of $V_{\widetilde{\mu}}$.

\begin{subsection}{A dimension formula}
\begin{corollary}\label{dimensionFormula}
Let $V_r$ be the double determinantal variety associated to the $m \times n$ matrices $X_1, \ldots, X_{r}$ of distinct indeterminates and minor sizes $s$ and $t$ as in Theorem \ref{generalCase}.  Call the defining ideal $J_r$ with $2 \leq r$.  Assume without loss of generality that $s \leq t$.  Then \begin{align*}
\height(J_r) &= [(m-s+1)(n-s+1)]+[(r-1)(n)(m-s+1)]\\
&+[(n-t+1){\sum_{q=2}^r }(s-1)-\operatorname{max}\{0,t-(s-1)(q-1)-1\}].
\end{align*}  In particular, considering $V_r$ as a projective variety, we have \ \begin{align*}
\dim(V_r) = (s-1)(m+n-s+1)&+[(t-1) {\sum_{q=2}^r} (s-1)-\operatorname{max}\{0,t-(s-1)(q-1)-1\}]\\
&+[n {\sum_{q=2}^r} \operatorname{max}\{0,t-(s-1)(q-1)-1\}].
\end{align*}
\end{corollary}
\begin{proof}
Following the induction outlined in Theorem \ref{generalCase} by the path of Young diagrams with $k = 1$ and $\lambda^0(0) = (2n, \ldots, 2n)$, $\lambda^0(1) = (2n, \ldots, 2n, 2n-1)$, $\ldots$, $\lambda^0(mn(r-1)) = (n, \ldots, n)$, where $\lambda^0(a)$ denotes the Young diagram we obtain after $a$ removals of the southernmost entry from easternmost column of the final matrix, i.e., after $a$ stages of the induction outlined in Theorem \ref{generalCase} always choosing to transition from the matrix $J$ to the matrix $N$.  Here $J_r = I_{\lambda^0(0)}(s,t)$.  

Call the variable being removed at the $w^{th}$ stage $x^q_{i,j}$.  Notice that if $i<s$, then there are no $s$-minors of $H_{\lambda^0(w)}$ having $x^q_{i,j}$ as the final entry along its main diagonal.  Similarly, if $1<q$ and $(s-1)(q-1)<t-i$, then every $t$-minor with $x^q_{i,j}$ as the final entry along its main diagonal in $V_{\lambda^0(w)}$ is in the ideal generated by the $s$-minors of $H_{\lambda^0(w)}$ because every such $t$-minor must involve at least $s$ rows from some matrix $X_{q'}$ with $q' < q$.  Moreover, if $j<t$, then there are no $t$-minors of $V_{\lambda^0(w)}$ involving $x^q_{i,j}$.  Therefore, if both $i<s$ and also $i<t-(s-1)(q-1)$ and $j<t$, then $I_{\lambda^0(w+1)}(s,t) = I_{\lambda^0(w)}(s,t)$.  Otherwise, $\height(I_{\lambda^0(w)}(s,t))=\height(I_{\lambda^0({w+1})}(s,t))+1$, as discussed in the proof of Theorem \ref{generalCase}.  In particular, $J_r$ has height $(r-1)(m-s)+(n-t+1)\sum_{q=2}^r \operatorname{max}\{0,m-(t-(s-1)(q-1))+1\}$ greater than that of $I_{\lambda^0(mn(r-1))}(s,t)$.  But $I_{\lambda^0(mn(r-1))}(s,t)$ is simply the determinantal ideal of the $s$-minors of the matrix $X_1$, whose height is $(m-s+1)(n-s+1)$.  Therefore, \begin{align*}
\height(J_r)&= \height I_{\lambda^0(mn(r-1))(s,t)}+(r-1)(m-s)+{\sum_{q=2}^r} (s-1)-\operatorname{max}\{0,t-(s-1)(q-1)-1\}\\
&= [(m-s+1)(n-s+1)]+[(r-1)(n)(m-s+1)]\\
&+[(n-t+1){\sum_{q=2}^r} (s-1)-\operatorname{max}\{0,t-(s-1)(q-1)-1\}].
\end{align*}  The dimension claim is then immediate from simple algebra.
\end{proof}
\end{subsection}

One could alternatively count $\dim(V_r)$ as the dimension of the variety of $s$-minors in one $m \times n$ matrix of indeterminates plus the number of indeterminates $x^q_{i,j}$ with $q>1$ not appearing as the final element of the leading term of any of the degree $s$ generators of $J$ or any of the degree $t$ generators in the ideal generated by the degree $s$ generators.  

\bigskip

\begin{example}
If $r = m = n = t = 3$ and $s = 2$, denote by $X$, $Y$, and $Z$ the matrices used to define $V_3$.  Call their horizontal concatenation $H$ and vertical $V$.  We compute $\dim(V_3) = (1)(5)+2+3-1$, where $5-1$ is the dimension of the (projective) variety of $2$-minors in $X$ and $2+3$ counts $z_{11}$, $z_{12}$, $y_{11}$, $y_{12}$, and $y_{13}$ as the variables not appearing as the final element of any main diagonal of a $2 \times 2$ submatrix of $H$ or a $3 \times 3$ submatrix of $V$ that is not in the ideal of $2$-minors of $H$.

  \begin{equation*}
    V = \begin{tikzpicture}[baseline=(current  bounding  box.center)]
      \matrix [matrix of math nodes,left delimiter=(,right delimiter=)] (m)
      {
        x_{11} & x_{12} & x_{13} \\ 
        x_{21} & x_{22} & x_{23} \\
        x_{31} & x_{32} & x_{33} \\
        \mathbf{y_{11}} & \mathbf{y_{12}} & \mathbf{y_{13}} \\
        y_{21} & y_{22} & y_{23} \\
        y_{31} & y_{32} & y_{33} \\
         \mathbf{z_{11}} & \mathbf{z_{12}} & z_{13} \\
        z_{21} & z_{22} & z_{23} \\
        z_{31} & z_{32} & z_{33} \\
      };
    \end{tikzpicture}
  \end{equation*}
\end{example}

\begin{corollary}
Double determinantal varieties are normal and irreducible.
\end{corollary}

\begin{proof}
Let $V_r$ be the double determinantal variety corresponding to the matrices $X_1, \ldots, X_r$ and rank restrictions $s$ and $t$ as in Theorem \ref{generalCase}, and call its homogeneous coordinate ring $S = R/J$, where $J$ is the ideal generated by the $s$-minors of $H$, the horizontal concatenation of $X_1, \ldots, X_r$, together with the $t$-minors of $V$, the vertical concatenation of the same.  Recall that we have defined double determinantal varieties only over a perfect field, and notice that the Jacobian ideal of $J$ is generated by the $(s-1)$-minors of $H$ together with the $(t-1)$-minors of $V$.  It follows from Corollary \ref{dimensionFormula} that the height of this Jacobian ideal is strictly greater than $1$ in $R/J$, and so $R/J$ is regular in codimension $1$.  Because $R/J$ is Cohen-Macaulay by Theorem \ref{generalCase}, it satisfies Serre's condition $S_2$.  But a ring that is both $R_1$ and $S_2$ is normal.  It is immediate the a graded ring that is normal is a domain, which is to say that $V_r$ is irreducible.  
\end{proof}


\begin{thebibliography}{99}

\bibitem{MS04} E. Miller and B. Sturmfels.  \emph{Combinatorial commutative algebra}, Springer Science \& Business Media 227, 2004.

\bibitem{CLO94} D. Cox, J. Little and D. O'Shea. \emph{Ideals, varieties, and algorithms}, American Mathematical Monthly 101 (6), 1994.
  
\bibitem{Con95}  A. Conca. \emph{Ladder determinantal rings}, J. Pure Appl. Algebra 98 (1995) 119--134.

\bibitem{CDS} A. Conca, E. De Negri, and Z. Stojanac, \emph{A characteristic free approach to secant varieties of triple Segre products}, Algebr. Comb. 3 (2020), no. 5, 1011--1021.

\bibitem{Gin08} V. Ginzburg. \emph{Lectures on Nakajima's quiver varieties}, arXiv:0905.0686.

\bibitem{GM00} N. Gonciulea and C. Miller. \emph{Mixed ladder determinantal Varieties}, J.  Algebra 231 (2000) 104--137..

\bibitem{Gor07} E. Gorla.  \emph{Mixed ladder determinantal varieties from two-sided ladders}, J. Pure Appl. Algebra 211 (2) (2007) 433--444.

\bibitem{GMN13} E. Gorla, J.C. Migliore and U. Nagel. \emph{Gr{\"o}bner bases via linkage}, J. Algebra 384 (2013) 110--134.

\bibitem{Har07} R. Hartshorne, \emph{Generalized divisors and biliaison}, Illinois J. Math. 51 (2007), no. 1, 83--98.

\bibitem{HE71} M. Hochster  and J. A. Eagon, \emph{Cohen--Macaulay rings, invariant theory, and the  generic perfection of determinantal loci},  Amer. J. Math. (93) (1971) 1020--1058.

 \bibitem{HT92} J. Herzog and N.V. Trung. \emph{Gr\"obner bases and multiplicity of determinantal and pfaffian ideals}, Adv. in Math. 96 (1992) l--37. 

\bibitem{IL} J. Illian and L. Li, \emph{Nakajima's quiver varieties and triangular bases of cluster algebras}, in preparation.

\bibitem{KMM+01} J. O. Kleppe, J. C. Migliore, R. M. Mir\'o-Roig, U. Nagel, and C. Peterson, \emph{Gorenstein liaison,
complete intersection liaison invariants and unobstructedness}, Mem. Amer. Math. Soc. 154 (2001),
no. 732.

\bibitem{Mat89} H. Matsumura. \emph{Commutative ring theory}.  Vol. 8. Cambridge university press, 1989.

\bibitem{MN02} J. Migliore and U. Nagel.  \emph{Liaison and related topics: Notes from the Torino workshop/schoo}, arXiv:math/0205161.

\bibitem{Nak16} H. Nakajima. \emph{Introduction to quiver varieties--for ring and representation theoriests}, arXiv:1611.10000.

 \bibitem{Nar86} H. Narasimhan. \emph{The irreducibility of ladder determinantal varieties}, J. Algebra 102 (1986) 162--185. 

\end{thebibliography}

\end{document}